\documentclass{amsart}

\usepackage{amssymb}
\usepackage{amsfonts}

\usepackage{ amsmath,amssymb,amsbsy,amsfonts, latexsym,amsopn,amstext, amsxtra,euscript,amscd}

\usepackage{cite}

\usepackage{amsrefs}

\usepackage{url}
\newcommand\myurl[1]{\url{#1}}

\BibSpec{webpage}{%
  +{}{\PrintAuthors} {author}
  +{,}{ \textit} {title}
  +{}{ \parenthesize} {date}
  +{,}{ \myurl} {myurl}
  +{,}{ } {note}
  +{.}{ } {transition}
}

\usepackage{xy} 
\input xy \xyoption{all}

\newtheorem{thm}{Theorem}
\newtheorem{proposition}{Proposition}
\newtheorem{lem}{Lemma}
\newtheorem{rem}{Remark}
\newtheorem{defn}{Definition}

\def\Z{\mathbb Z}

\def\C{\mathbb C}
\def\P{\mathbb P^1 (k)}

\def\Z{{\mathbb Z}}
\def\C{{\mathbb C}}

\def\L{\mathcal L}
\def\H{\mathcal H}
\def\M{\mathcal M}
\def\X{\mathcal X}
\def\S{\mathcal S}

\def\p{\mathfrak p}
\def\u{\mathfrak u}

\def\ss{\mathfrak s}
\def\u{\mathfrak u}
\def\u{\mathfrak s}
\def\p{\mathfrak p}
\def\s{\mathfrak s}

\def\D{\Delta}

\def\emb{\hookrightarrow}

\def\l{\lambda}
\def\a{\alpha}
\def\e{\varepsilon}

\def\t{\tau}
\def\<{\langle}
\def\>{\rangle}
\def\zz{\zeta}
\def\z{\omega}

\def\iso{{\, \cong\, }}

\def\bG{\bar G}

\def\bG{{\bar G}}

\def\Z{\mathbb Z}

\def\C{\mathbb C}

\def\C{\mathcal C}
\def\H{\mathcal H}
\def\M{\mathcal M}

\def\l{\lambda}
\def\a{\alpha}

\def\p{\mathfrak p}

\def\Z{\mathbb Z}

\def\C{\mathbb C}
\def\L{\mathcal L}
\def\H{\mathcal H}
\def\M{\mathcal M}
\def\l{\lambda}
\def\s{\sigma}
\def\a{\alpha}
\def\p{\mathfrak p}
\def\u{\mathfrak u}

\def\<{\langle}
\def\>{\rangle}

\def\emb{\hookrightarrow}

\def\D{\Delta}

\def\u{\mathfrak u}
\def\L{\mathcal L}

\def\S{\mathcal S}

\def\P{\mathbb P^1 (k)}

\def\Z{{\mathbb Z}}
\def\C{{\mathbb C}}
\def\M{{\mathcal M}}
\def\H{\mathcal H}
\def\L{\mathcal L}

\def\Aut{\mbox{Aut }}
\def\bAut{\overline {\mathrm{Aut}}}

\def\D{\Delta}
\def\e{\varepsilon}
\def\u{\mathfrak s}

\def\bG{\bar G}
\def\t{\tau}

\def\a{{\alpha }}

\def\<{\langle}
\def\>{\rangle}

\def\s{\mathfrak s}

\def\p{\mathfrak p}

\def\emb{\hookrightarrow }

\def\zz{\zeta}

\def\z{\omega}

\begin{document}

\title{Bielliptic curves of genus 3 in the hyperelliptic moduli}

\author{T. Shaska         \and        F. Thompson }

\maketitle

\begin{abstract}
In this paper we study bielliptic curves of genus 3 defined over an algebraically closed field $k$ and the intersection of the moduli space $\M_3^b$ of such curves with the hyperelliptic moduli $\H_3$.  Such intersection $\S$   is an irreducible,  3-dimensional, rational  algebraic variety.  We determine the equation of this space in terms of the $Gl(2, k)$-invariants of binary octavics as defined in \cite{hyp_mod_3}   and find a birational parametrization of $\S$.  We also compute all possible subloci of curves for all possible automorphism group $G$.   Moreover, for every rational moduli point  $\p \in \S$, such that $| \Aut (\p) | > 4$, we give explicitly a rational model of the corresponding  curve over its field of moduli in terms of the $Gl(2, k)$-invariants.
%
\end{abstract}

\section{Introduction}
The moduli space $\M_g$  of algebraic curves of genus $g\geq 2$, defined over an algebraically closed field $k$,  is an interesting object that has received plenty of attention since the mid XX-century.  It is an irreducible  quasi-projective variety of dimension $3g - 3$. Understanding the stratification of this space has been also a major problem with many papers written on the subject to this day.  There are two main difficulties on this problem:   

\emph{i)  an explicit description of $\M_g$ is not known (i.e., a coordinate in $\M_g$),}

\emph{ii) a list of automorphism groups for a fixed $g\geq 2$ has  not been known. }  \\

\noindent    Naturally one has a better chance to address the above problem if focused on the hyperelliptic sublocus $\H_g$ of $\M_g$, since it is easier to pick a coordinate on the space $\H_g$.  After all, the hyperelliptic curves were well understood since the XIX-century and restricting the problem to the hyperelliptic locus seems reasonable.  It was well known to classical algebraic geometers of the XIX-century that the isomorphism classes of hyperelliptic curves defined over an algebraically closed field $k$ correspond to the orbits of the $GL_2 (k)$ action on the space of binary forms of degree $2g+2$ with coefficients from $k$.  This was, among others, one of the main motivations of the invariant theory during the XIX-century. For the generalization to the case of superelliptic curves one can check \cite{super2}. 

The case of genus 2 had been studied extensively by XIX-century mathematicians; see \cite{Burkhardt1, Burkhardt2} even though the concept of the moduli space was not quite refined at the time.  
About a decade ago Gaudry/Schost in \cite{gaudry} attacked the problem for $g=2$ from the computational point of view. After all, a coordinate in $\M_2$ could be fixed using the Igusa invariants and the list of automorphism groups of genus 2 was known; see \cite{Ge} among others.  At the same time that \cite{gaudry} was being circulated as a preprint, Shaska/V\"olklein \cite{sh-v} considered the problem from a more group-theoretical point of view.  Of course, the main case in both those papers was the case when the genus 2 curves had an elliptic involution.  The locus $\L_2$ of such curves is a 2-dimensional irreducible variety in $\M_2$  computed in both papers. In the process,  a group action was discovered in \cite{sh-v}  and its $u, v$ invariants were instrumental in computing equations for the strata of $\M_2$.  The map  
\begin{equation}\label{dihedral}
 \phi_2 (u, v) \to (i_1, i_2, i_3) 
 \end{equation} 
provides a birational parametrization of the space $\L_2$, where $i_1, i_2, i_3$ are $GL_2 (k)$-invariants in the space of binary sextics.   The singular locus of this map correspond to genus 2 curves with larger automorphism group; see \cite[Lemma 3]{sh-v}. The paper \cite{sh-v} spurred interest in two directions.  First, it naturally brought to the attention of the authors  the problem of automorphism groups of curves of genus $g\geq 2$.  This corresponds to the ii) part of the problem stated in the beginning. Second, naturally raised the question whether the invariants $u, v$ for $g=2$ could be generalized to higher genus.  In the next two paragraphs we consider each direction in more details.  

Determining the list of automorphism groups of algebraic curves of genus $g\geq 2$ is a classical problem.  There were hundreds of papers in the subject before 2001, most of them considering specific cases for small genus.  However, there was one interesting development at the time that it seems as it did not get the attention  it deserved. Breuer computed all signatures of the groups acting on compact Riemann surfaces for genus $g \leq 48$;   see \cite{breuer}.  The restriction $g \leq 48$ is merely technical and Breuer's algorithm   works for any genus, providing that some careful analysis is required for sporadic cases.  Using results in \cite{breuer} and the theory of Hurwitz spaces,  Magaard, Shpectorov, Shaska, V\"olklein determined an algorithm of how to determine the list of full automorphism  groups of curves for any given genus $g\geq 2$; see \cite{kyoto}.   In \cite{kyoto}  a complete list of full automorphism groups for curves of genus 3 was determined and the corresponding equations were provided as a way to illustrate the methods described in that paper.  There were tens of papers on the case of genus $g=3$ before \cite{kyoto} appeared. Moreover, by methods in \cite{kyoto}  the list of full automorphism groups of curves for any genus $g \geq 2$ can be determined. This settles the second part ii) of the initial problem.

The second direction that was spurred by \cite{sh-v} was the problem of generalizing the map \eqref{dihedral} to higher genus.  Natural questions to follow would be whether the curves with larger automorphism groups would be in the singular locus of $\phi$.  The group action discovered in \cite{sh-v} was generalized in \cite{sh_05} and then in a more formal paper in \cite{g-sh} were such invariants in higher genus were called \textbf{dihedral invariants}. For a genus $g\geq 2$ now we have a map
\begin{equation}
 \phi_g \,  (\s_4, \dots , \s_g) \to \left(t_1, \dots , t_{2g-1} \right),
 \end{equation}
where $t_1, \dots , t_{2g-1}$ are $GL_2 (k)$-invariants in the space of binary forms of degree $2g-1$. 
In papers \cite{issac}, \cite{g-sh-s}, and \cite{g-sh} the case of stratification of the hyperelliptic moduli $\H_3$ was treated to illustrate the general theory.  Dihedral invariants have been used quite extensively since by many authors. They were generalized for fields of positive characteristic by \cite{AK} and are defined again in the projective version in \cite{ritz2},  where they are renamed  as  \textit{dihedral arithmetic invariants}.

This paper takes another look at the study of genus 3 hyperelliptic curves with extra automorphisms.  The general strategy for $g=3$   was quite obvious a decade ago; one computes the Shioda invariants defined in \cite{shi1} starting from the Table 1 of \cite{kyoto} and then eliminating the parameters which appear as coefficients of the curve.   Such computations were simplified considerable via  the dihedral invariants $\s_2, \s_3, \s_4$. There was an obvious drawback compared to the genus $g=2$ case; the $GL_2 (k)$-invariants for $g=3$ were not known.  Hence, the obvious strategy was to describe the strata in terms of the $SL_2 (k)$-invariants defined by Shioda in \cite{shi1}.  This approach is taken in  \cite{ritz1}.  It is not clear from \cite{ritz1} if the dihedral invariants were used in these computations or they were performed straight from the  equations of the curves as in Table 3 in \cite{kyoto}.   In any case, Shioda invariants $J_2, \dots , J_{10}$ are computed in \cite{ritz1} and using syzygies determined in \cite[Theorem 5]{shi1} the authors determine each loci in $\H_3$.  In \cite{hyp_mod_3} it was shown that the syzygies determined by Shioda in \cite[Theorem 5]{shi1} are not correct.   It is unclear if the authors in \cite{ritz1} have corrected such  syzygies, otherwise all the results of \cite{ritz1} could be incorrect.  

The motivation for this paper was the definition of $GL_2(k)$-invariants in \cite{hyp_mod_3} where an explicit equation of the hyperelliptic moduli  $\H_3$ is given in the ambient space $\C^6$.  Using the absolute invariants $t_1, \dots , t_6$ as in \cite{hyp_mod_3} and the dihedral invariants $\s_2, \s_3, \s_4$ one can easily compute the locus in $\H_3$ for each case of Table 3 of \cite{kyoto}.  The drawback of invariants $t_1, \dots , t_6$ is that they are not defined everywhere.  However, this is done by choice so that their degrees are kept small.  This makes computations a lot easier.  One can get projective equations (i.e., equations in terms of $J_2, \dots , J_8$) from our equations simply by replacing $t_1, \dots , t_6$ with their definitions and clearing denominators.

The paper is organized as follows. In section 2 we give a brief description of the invariants of the binary octavics and definitions of absolute invariants.  Notice that our definitions of $J_2, \dots , J_8$ are slightly different from those of Shioda.  In section 3 we discuss genus 3 hyperelliptic curves with an elliptic involution and derive an parametric equation for such family of curves.  This is rather known material that has been treated in \cite{sh_05, g-sh}.  

Section 4 is the main section of the paper where the equation for the locus of curves with an elliptic involution.  The main theorem here describes the equation of the irreducible sub variety $\S$ in the hyperelliptic moduli $\H_3$.  We show that for a generic curve $C$ in $\S$ the field of definition is at most a degree 2 extension of the field of moduli.  This is an improvement from \cite{ritz1} where it is shown that this bound is eight.  For example, for the case when the group is $V_4$, we get a model for the parametric curve defined over a quadratic extension of the field of moduli versus a degree 8 extension in \cite{ritz1}.

In section 5, we determine the equation of all 1 and 2-dimensional loci for any fixed automorphism group $G$.  Parametrization of such loci were also given in \cite{g-sh-s}. Here we compute them in terms of the absolute invariants $t_1, \dots , t_6$.  

The goal of this paper was to describe the stratification of the space $\H_3$ in terms of the absolute invariants $t_1, \dots , t_6$.  The benefit of this approach is that there are fewer equations and even simpler ones.  The results in \cite{hyp_mod_3} make it possible that we do not have to use the invariants $J_9,  J_{10}$ and have fewer equations in each case.  We get better results compared to \cite{ritz1} in the case of the group $V_4$ and $\Z_2^3$ on the minimal equation of the curves over their field of moduli.  

In the case of group $\Z_2^3$ we prove that the field of moduli is a field of definition and give a model of the curve over its field of definition.  Some of these results were not new to us, since they were proved in \cite{g-sh}.  However, in this paper we are able to explicitly describe such results in terms of the absolute invariants $t_1, \dots , t_6$. All our results are implemented in a Maple package which is provided for free on \cite{homepage}.

\medskip

\noindent \textbf{Notation:}  Throughout this paper, by a "curve" we mean an irreducible algebraic curve defined over an algebraically closed field  $k$.  While we use invariants $J_2, \dots , J_8$ of binary octavics as Shioda \cite{shi1}, the reader must be aware that our definitions are not the same as those in Shioda's paper, instead we use the definitions as \cite{hyp_mod_3}. We also use the dihedral invariants $\s_2, \s_3, \s_4$ which are the same as those used in \cite{g-sh-s} $u, v, w$,  where $\s_4=u$, $\s_3=v$, $\s_2= w$.


%

\section{Bielliptic genus 3 curves}
Let $k$ be an algebraically closed field of characteristic zero and $\X$ an irreducible, smooth, projective curve of genus $g \geq 3$ defined over $k$. As usual, we denote by $\M_g$ the coarse moduli space of smooth curves of genus $g\geq 2$ and by $\H_g$ the hyperelliptic locus in $\M_g$. The isomorphism class of $C$, i.e. the corresponding point in $\M_g$, is denoted by $[C]$.  

A curve $C$ is called \textit{bielliptic} if it admits a degree 2 morphism $\pi : C \to E$ onto an elliptic curve.  Let 
\[ \M_g^b = \{ [C] \in \M_g \, : \, C \, \, \textit{bielliptic } \} \]
be the locus of bielliptic curves in $\M_g$.   $\M_g^b$ is an irreducible $(2g-2)$-dimensional sub variety of $\M_g$.  For $g=3$ $\M_3^b$ is the unique component of maximal dimension of the singular locus $Sing (\M_3)$; see \cite{cornalba0, cornalba1} for details. It is known that 

i) $\M_3^b$ is rational

ii) $\M_3^b \cap \H_3$ is an irreducible, codimension 1, rational subvariety of $\M_3^b$,

\noindent see \cite[Theorem 1.1]{BC} for details.  In this paper we aim to find an algebraic equation for the space $\S := \M_3^b \cap \H_3$. An algebraic equation for $\M_3^b$ using invariants of ternary quartics and a theorem of Kovalevskaja is intended in \cite{dolga}.

Let $\a$ be the element in $\Aut (C)$ which interchanges the sheets of $\pi : C \to E$ such that $E \iso C/\< \a \>$.  We call $\a$ the \textbf{elliptic involution} of $C$ corresponding to $\pi$. Hence, the space $\S= \M_3^b \cap \H_3$ is exactly the space of genus 3 hyperelliptic curves with elliptic involutions.  Such space has beed studied before from the point of view of automorphism groups,  as described in details in the introduction.


For a fixed group $G$ acting on a genus $g$ algebraic curves $\X_g$ we have a covering $\X_g \to \X_g/G$.  All possible  ramification structures of such covering for any genus $g$ hyperelliptic curves were determined in \cite{serdica}.  Indeed, this is also done for all superelliptic curves; see \cite{beshaj-2} for details.  In the case of hyperelliptic curves of genus 3, each group occurs only with one signature; see \cite{kyoto}.  Hence,  there is no confusion if we denote by $\S (G)$ the locus in $\H_3$ of all curves with automorphism group isomorphic to $G$ (i.e., $G \emb \Aut( \X_g)$).  The loci $\S (G)$ is not a priori irreducible.    In general, irreducibility is checked by  the braid action on Nielsen tuples. The locus $\S ( G)$ is a Hurwitz space of covers with monodromy group  $G$ and fixed ramification structure.  
However, under our assumptions ($g=3$ and hyperelliptic) this is always the case as shown in  \cite{serdica} and we will avoid that discussion here.

\section{Genus 3 hyperelliptic fields with elliptic involutions}

Let $K$ be a genus 3 hyperelliptic field. Then $K$ has exactly one genus 0 subfield of degree 2, call it $k(X)$.  It is the fixed field of the \textbf{hyperelliptic involution} $\z_0$ in $\Aut (K)$. Thus, $\z_0$ is central in $\Aut (K)$, where  $\Aut (K)$ denotes the group $\Aut (K/k)$. It induces a subgroup of $\Aut (k(X))$ which is naturally isomorphic to $\bAut (K):= \Aut (K)/\<\z_0\>$. The latter is called the \textbf{reduced automorphism group} of $K$.


If $\z_1$ is a non-hyperelliptic involution in $G$ then $\z_2:=\z_0\, \z_1$  is another one. So the   non-hyperelliptic involutions come naturally in (unordered) pairs $\z_1$, $\z_2$. 
These pairs correspond  bijectively to the Klein 4-groups in $G$. Indeed,  each Klein 4-group in $G$ contains $\z_0$.

\begin{defn}
We will consider pairs $(K, \e)$  with $K$ a genus 3 hyperelliptic  field and $\e$ an non-hyperelliptic involution in $\bG$. Two such pairs $(K,\e)$ and  $(K', \e')$ are called isomorphic if there is a $k$-isomorphism  $\a: K \to K'$ with $\e' = \a \e \a^{-1}$.
\end{defn}

Let  $\e$ be an non-hyperelliptic involution in $\bG$. We can choose the generator $X$ of $\mbox{Fix}(\z_0)$      such that $\e(X)=-X$.  Then $K=k(X,Y)$ where $X, Y$ satisfy equation
\[ Y^2 = (X^2-\a_1^2) (X^2-\a_2^2) (X^2-\a_3^2) (X^2-\a_4^2)  \]
for some $\a_i \in k$, $i=1, \dots, 4$.  Denote by
\begin{equation}
\begin{split}
s_1 = & - \left( \a_1^2 + \a_2^2 + \a_3^2 + \a_4^2 \right) \\
s_2 = & \,  (\a_1 \a_2)^2 + (\a_1 \a_3)^2 + (\a_1 \a_4)^2 + (\a_2 \a_3)^2 + (\a_2 \a_4)^2 + (\a_3 \a_4)^2 \\
s_3 = & -   ( \a_1\,\a_2\,\a_3 )^2 - (\a_4\,\a_1\,\a_2)^2 - (\a_4\,\a_3\,\a_1)^2 - (\a_4\,\a_3\,\a_2)^2 \\
s_4 = & - \left( \a_1 \a_2 \a_3 \a_4 \right)^2\\
\end{split}
\end{equation}
Then,  we have 
\[ Y^2=X^8+s_1 X^6+ s_2 X^4+ s_3 X^2+s_4  \]
with $s_1, s_2, s_3, s_4 \in k$, $s_4\ne 0$. Further $E_1=k(X^2,Y)$ and $ C=k(X^2, YX)$ are the two  subfields corresponding to $\e$ of genus 1 and 2 respectively.   

\par Preserving the condition $\e(X)=-X$ we  can further modify $X$ such that  $s_4=1$. Then, we have the following:

\begin{lem}  Every genus 3 hyperelliptic curve $\X$,  defined over a field $k$, which has an non-hyperelliptic involution has equation 
\begin{equation}\label{eq3}
Y^2=X^8+a X^6+ b X^4+ c X^2+1
\end{equation}
for some $a, b, c \in k^3$,  where the polynomial on the right has non-zero discriminant.
\end{lem}

Indeed, the non-hyperelliptic involution above is an elliptic involution and $\X$ is bielliptic.  There is another non-hyperelliptic involution of $\X$, as noted above, namely $\z_2:=\z_0\, \z_1$ which fixes a genus 2 field.  See \cite{sh2} for the equation of this genus 2 subfield and the arithmetic of such curves. hence, we have the following result; see \cite{BC} or \cite{sh2} for details. 

\begin{proposition}
$[C] \in \S $ if and only if $C$ is a double covering of a genus 2 curve.
\end{proposition}

The above  conditions  determine $X$ up to coordinate change by the group $\< \tau_1, \tau_2\>$ where  
\[ \tau_1: X\to \zz_8 X, \quad \textit{and } \quad   \tau_2: X\to \frac 1 X,\] and  $\zz_8$ is a primitive 8-th root of unity in  $k$. 
Hence,  
\[\tau_1 : \, (a, b, c) \to (\zz_8^6 a, \zz_8^4 b, \zz^2 c ),\]  
and 
\[ \tau_2 : \, (a, b,  c) \to (c, b, a) .\]
Then, $| \tau_1 | =4$ and $|\tau_2 | =2$.  The group generated by $\tau_1$ and $\tau_2$ is the dihedral group of order 8.   Invariants of this action are
\begin{equation}\label{eq2}
\begin{split}
\s_2  & =\,a\,c,  \\
 \s_3  & =(a^2+c^2)\,b,  \\
  \s_4  & =a^4+c^4, 
\end{split}
\end{equation}
since 
\[ 
\begin{split}
& \tau_1 (a^4+c^4) = (\zz_8^6 a)^4 + (\zz_8^2 c)^4 = a^4 + c^4 \\
& \tau_1 \left(   (a^2 + c^2 ) b \right) = \left( \zz_8^4 a^2 + \zz_8^4 c^2   \right) \cdot ( \zz_8^4 b) = (a^2+c^2) b \\
& \tau_1 ( a c ) =  \zz_8^6 a \cdot \zz_8^2 c =  a c \\
\end{split}
\] 
Since they are symmetric in $a$ and $c$, then they are obviously invariant under $\t_2$. 
Notice that $\s_2, \s_3, \s_4$ are homogenous polynomials of degree 2, 3, and 4 respectively.  The subscript $i$ represents the degree of the polynomial $\s_i$.

Since the above transformations are automorphisms of the projective line $\mathbb P^1 (k)$ then the $SL_2(k)$ invariants must be expressed in terms of $\s_4, \s_3$, and $\s_2$. 
In these parameters, the discriminant of the octavic polynomial on the right hand side of  Eq.~ \eqref{eq3} equals $- \frac {256} {(\s_4 + 2 \s_2^2)^4} \,  \Delta^2 $, where
%
\begin{equation}\label{D_3}
\begin{split}
\Delta  & =     132 {\s_2}^4 \s_4  -18 {\s_4}^2 \s_2 \s_3-72 \s_4 {\s_2}^3 \s_3-\s_4 {\s_2}^2 {\s_3}^2+80 \s_2 {\s_3}^2\s_4-576 \s_3 {\s_2}^2\s_4\\
& -256 {\s_4}^2+768 \s_4 {\s_2}^3-1024 \s_4{\s_2}^2+256 {\s_2}^2{\s_3 }^2-576 {\s_2}^4\s_3+768 {\s_2}^5+24 {\s_2}^{6}\\
& -16 {\s_3}^4 -1024 {\s_2}^4+128 {\s_3}^2\s_4+192 {\s_4}^2\s_2+114 {\s_4}^2{\s_2}^2+4 {\s_4}^2{\s_2}^3-144 {\s_4}^2\s_3\\
& +16 \s_4 {\s_2}^5-72 {\s_2}^5 \s_3-2 {\s_2}^4 {\s_3}^2+160 {\s_2}^3{\s_3}^2+4 {\s_3}^3\s_4+8 {\s_3}^3{\s_2}^2+27 {\s_4}^3+16 {\s_2}^{7}
\end{split}
\end{equation}
%
%
\noindent The map \[(a, b, c) \mapsto (\s_2, \s_3, \s_4)\] is a branched Galois covering with group $D_4$  of the set 
\[ \{  (\s_2, \s_3, \s_4)\in k^3 : \Delta_{(\s_2, \s_3, \s_4)} \neq 0\}\]
by the corresponding open subset of $a, b, c$-space. In any case, it is true that if $a, b, c$ and $a', b', c'$ have the same $\s_2, \s_3, \s_4$-invariants then they are conjugate under $\< \tau_1, \tau_2\>$.

The case when $\s_3 =0$ must be treated separately.  We have two sub cases $a^2+c^2=0$ or $b=0$.  
Then we define new invariants as follows:
\begin{equation} \label{def_u}
\begin{split}
\p (\X_3) \, = \left\{ \aligned & w=b^2 & \mathrm{if} \quad a=c=0,\\
 & (\s_2, w, \s_4) & \mathrm{if} \quad a^2+c^2=0\  \mathrm{and} \, \, b\neq 0,\\
 & (\s_2, \s_3, \s_4) & \mathrm{otherwise}. \\ \endaligned
\right.
\end{split}
\end{equation}
%
The invariants $\s_2, \s_3, \s_4, \dots $ are valid for any genus $g \geq 2$ and are called by many authors \textbf{dihedral invariants}.  They were discovered by the second author in his PhD thesis and appeared for the first time in the literature in Shaska/V\"olklein \cite{sh-v}.   Then, they appeared for genus $g=3$ in \cite{issac, g-sh-s} and were generalized for every genus  in \cite{g-sh}. They were  generalized a ditto to all cyclic curves by Antoniadis/Kontogiorgis \cite{AK}. In \cite{ritz2} a projective version of these dihedral invariants are called \textit{dihedral arithmetic invariants}. 

\begin{lem}\label{lemma1}  For $(a, b, c) \in k^3$ with $\Delta\neq 0$, equation \eqref{eq3} defines a  genus 3 hyperelliptic  field  $K_{a, b, c}=k(X,Y)$. 
Its reduced automorphism group contains the elliptic involution $\e_{a, b, c}: X \mapsto -X$. Two such  pairs $(K_{a, b, c}, \e_{a, b, c})$ and  $(K_{a', b', c'}, \e_{a', b', c''})$ are isomorphic if and only if 
\[ \left( \s_2, \s_3, \s_4 \right) = \left( \s_2', \s_3', \s_4' \right) \] 
where $\s_2, \s_3, \s_4$ and $\s_4',\s_3', \s_2'$ are dihedral invariants associated with $a, b, c$ and $a', b', c'$, respectively.
\end{lem}

\proof An isomorphism $\a$ between these two pairs yields $K=k(X,Y)=k(X', Y')$ with $ k(X)=k(X')$ such that $X,Y$ satisfy \eqref{eq3} and  $X',Y'$ satisfy the corresponding equation with $a, b, c$
replaced by $a', b', c'$. Further, $\e_{a, b, c}(X')=-X'$. Thus $X'$ is conjugate to $X$ under $\< \tau_1, \tau_2\>$ by the above remarks.
This proves the condition is necessary. It is clearly sufficient.

\qed

\begin{rem}
If $(2\s_4+\s_2^2) =0$, then this implies that $a=c=0$.  In this case the equation of the curve becomes \[ Y^2= X^8 + b X^4 + 1,\]  which corresponds to the curves with automorphism group $\Z_2 \times D_8$, (cf. Eq.~\eqref{Z2xD8})
\end{rem}

\noindent We have the following theorem

\begin{thm}\label{sect1_thm} Let $(\s_2, \s_3, \s_4) \in k^3\setminus\{ \D =0\}$.  Then the following hold:

i) The ordered triples $(\s_2, \s_3, \s_4)$  bijectively parameterize the isomorphism classes of pairs $(K,\e)$ where $K$ is a genus 3 hyperelliptic field and $\e$ an elliptic involution of $\Aut (K)$. The j-invariant of the  elliptic subfield of $K$ associated with $\e$ is given by  
\begin{equation} \label{j_eq}
j  = \frac {64} M \cdot \frac { (-4 \s_3^2-48 \s_4-24 \s_2^2+3 \s_2^3+6 \s_4 \s_2)^3} { (2 \s_4+\s_2^2) },  
\end{equation}
where 
\[
\begin{split}
M= & \,  \, 66 \s_4 \s_2^4  -2048 \s_4^2-512 \s_2^4-2048 \s_4 \s_2^2-128 \s_3^4+1024 \s_3^2 \s_4+512 \s_3^2 \s_2^2+228 \s_4^2 \s_2^2\\
  & +768 \s_4^2 \s_2+216 \s_4^3+\s_2^7+3 \s_2^6+4 \s_2^3 \s_4^2+4 \s_2^5 \s_4-\s_3^2 \s_2^4+768 \s_4 \s_2^3+160 \s_3^2 \s_2^3\\
  & +192 \s_2^5-2 \s_3^2 \s_4 \s_2^2+320 \s_3^2 \s_4 \s_2-72 \s_4^2 \s_3 \s_2-72 \s_4 \s_2^3 \s_3-1152 \s_4 \s_2^2 \s_3+32 \s_3^3 \s_4\\
  & -1152 \s_4^2 \s_3-288 \s_2^4 \s_3 +16 \s_3^3 \s_2^2   -18 \s_2^5 \s_3  
\end{split}
\]

ii) There is another involution $\z_0 \e \in \Aut (K)$ which fixes a genus 2 curve $\X_2$ with equation \[ Y^2=X(X^4+aX^3+bX^2+cX+1.\]
The isomorphism class of $\X_2$ is determined uniquely by the triple $(\s_2, \s_4, \s_4)$ as in Eq.~\eqref{i1_i2_i3}.

iii) The triples $(\s_2, \s_3, \s_4)$    parametrize the isomorphism classes of genus 3 hyperelliptic fields with $V_4 \emb \Aut (K)$. \\

\end{thm}

\proof   
i) The automorphism $\e \in \Aut (K) $ fixes a degree 2 elliptic subfield $E$ which has equation \[ Y^2 = x^4+ ax^3+bx^2 + cx +1 \] and $j$-invariant given in terms of $a$, $b$, and $c$.  Using substitutions in  Eq.~\eqref{subs} we get $j (E)$ as in Eq.\eqref{j_eq}.

ii)  The quotient $\X/\< \z_0 \e\> $ has genus 2.  This follows straight from the Riemann-Hurwitz formula.  The equation of this genus 2 curve is as claimed; see \cite{sh2}.  The isomorphism class of this genus 2 curve is determined by the absolute invariants $(i_1, i_2, i_3)$.  In terms of the $\s_2, \s_3, \s_4$ they have the following expressions
%
\begin{equation}\label{i1_i2_i3}
\begin{split}
i_1  = &   \frac {9 (2 \s_1+\s_3^2) }  {D^2} (\s_3^4 - 80 \s_3^2-72 \s_2^2 -2 \s_3^2 \s_1 - 24 \s_1 \s_2 - 12 \s_3^2 \s_2 + 2 \s_3^3 - 160 \s_1 + 4 \s_3 \s_1) \\ \\
i_2  = & \frac { 27 (2 \s_1+\s_3^2)^2}  {D^3} \left(2 \s_3^3 \s_1 -1116 \s_1 \s_2+\s_3^5-2240 \s_3^2 +162 \s_2^2 \s_3 +864 \s_2^2 +216 \s_1^2 \right. \\
& \left. +114 \s_3^2 \s_1+3 \s_3^4-558 \s_3^2 \s_2-4480 \s_1+624 \s_3^3 -18 \s_3^3 \s_2-36 \s_1 \s_2 \s_3+1248 \s_3 \s_1 \right)  \\ \\
i_3  =  & \frac {243} {1024} \, \frac { (2 \s_1+\s_3^2)^3 }  {  D^5  } \left(\s_3^7 -128 \s_2^4-2048 \s_1^2+768 \s_3^3 \s_1  -2048 \s_3^2 \s_1+192 \s_3^5-512 \s_3^4 \right.\\
& +216 \s_1^3+3 \s_3^6-72 \s_1^2 \s_3 \s_2+320 \s_1 \s_2^2 \s_3-72 \s_1 \s_3^3 \s_2 -1152 \s_1 \s_3^2 \s_2-2 \s_2^2 \s_1 \s_3^2  \\
& -1152 \s_1^2 \s_2+1024 \s_1 \s_2^2+160 \s_2^2 \s_3^3+512 \s_2^2 \s_3^2-18 \s_3^5 \s_2-288 \s_3^4 \s_2  +768 \s_3 \s_1^2 \\
& \left. +4 \s_3^3 \s_1^2 +4 \s_3^5 \s_1 +228 \s_3^2 \s_1^2+66 \s_3^4 \s_1+16 \s_2^3 \s_3^2+32 \s_2^3 \s_1-\s_2^2 \s_3^4 \right), \\
\end{split}
\end{equation}
%
where $D= -20 \s_1-10 \s_3^2+2 \s_3^3+4 \s_3 \s_1-3 \s_2^2$.

iii) This is a straight consequence of the first two parts. The cases when $| \Aut (K) | > 4$ are treated in Thm.~\ref{thm_last}.
\qed



\section{The locus $\S$ of genus 3 hyperelliptic curves with elliptic involutions}

In this section, first  we briefly define the    invariants of binary octavics.    We will use interchangeably the terms genus 3 hyperelliptic curve and genus 3 hyperelliptic field.  There is a one to one equivalence between the isomorphic classes of genus 3 hyperelliptic curves and projective classes of equivalence of binary octavics.  Thus, we have to describe some basic properties of binary octavics.  The following material can be found on works of classical algebraic geometers; see Alagna \cite{Al}, van Gall \cite{vG}, et al or for a modern version one can check \cite{dolga_book}.  

The ring of invariants of binary octavics was also studied by Shioda \cite{shi1}.  However, the syzygies among such $SL_2 (k)$-invariants described in the Shioda's paper seem to be incorrect.  In \cite{hyp_mod_3} such $SL_2 (k)$-invariants $J_2, \dots , J_8$ are redefined  and  the algebraic relations among them determined.  Furthermore, $GL_2(k)$-invariants $t_1, \dots , t_6$ are defined and relation among them determined.  Throughout this paper we will make use of these $GL_2 (k)$-invariants and therefore follow  definitions from \cite{hyp_mod_3}. 

Let $f(X,Y)$ be the binary octavic
$$f(X,Y) = \sum_{i=0}^8 a_i X^i Y^{8-i}.$$
defined over an algebraically closed field $k$.  We define the following covariants:
\begin{equation}
\begin{split}
&g=(f,f)^4, \quad k=(f, f )^6, \quad h=(k,k)^2, \\
& m=(f,k)^4, \quad  n=(f,h)^4, \quad p=(g,k)^4, \quad q=(g, h)^4,\\
\end{split}
\end{equation}
where the operator $( \cdot    , \cdot )^n$ denotes the $n$-th transvection of two binary forms; see \cite{hyp_mod_3} among many other references.

Then, the following 
\begin{equation}\label{def_J}
\begin{aligned}
&     J_2= 2^2 \cdot 5 \cdot 7 \cdot (f,f)^8,                    &   \qquad   &  J_3= \frac 1 3 \cdot  2^4 \cdot 5^2 \cdot 7^3 \cdot (f,g )^8, \\
&  J_4= 2^9 \cdot 3 \cdot 7^4 \cdot (k,k)^4,                      &    \qquad      &   J_5= 2^9 \cdot 5 \cdot 7^5 \cdot (m,k)^4,  \\
&   J_6 = 2^{14} \cdot 3^2 \cdot 7^6 \cdot (k,h )^4,              &    \qquad      &      J_7= 2^{14} \cdot 3 \cdot 5 \cdot 7^7 \cdot (m,h )^4,  \\
& J_8= 2^{17} \cdot 3 \cdot 5^2 \cdot 7^9 \cdot   (p,h)^4,        &    \qquad       &   J_9= 2^{19} \cdot 3^2 \cdot 5 \cdot 7^9 \cdot   (n,h)^4, \\
&  J_{10}=   2^{22} \cdot 3^2 \cdot 5^2 \cdot 7^{11} (q,h)^4      &  \qquad    &       \\
\end{aligned}
\end{equation}
are $SL_2(k)$- invariants; see \cite{hyp_mod_3} for details.     The following is a classical fact of invariant theory of binary forms.

\begin{lem}
Two binary forms $f(X, Y)$ and $f^\prime (X, Y)$ are projectively equivalent via  a matrix $M \in GL_2 (k)$ if and only if 
\[ J_i (f) = \lambda^i J_i (f^\prime), \quad where \quad \lambda = \left( \det (M) \right)^4 \]
\end{lem}

The following technical result is helpful for the rest of the paper; see \cite[Lemma 4]{hyp_mod_3}.

\begin{lem}
If $J_2=\dots = J_7=0$, then  the binary octavic has a double root.
\end{lem} 

Next, we define  $GL(2, k)$-invariants   as  follows
\[
t_1:= \frac {J_3^2} {J_2^3}, \quad 
t_2:= \frac {J_4} {J_2^2}, \quad 
t_3:= \frac {J_5} {J_2\cdot J_3}, \quad 
t_4:= \frac {J_6} {J_2\cdot  J_4}, \quad 
t_5:= \frac {J_7} {J_2 \cdot J_5}, \quad 
t_6:= \frac {J_8} {J_2^4},
\]
There is an algebraic relation 
\begin{equation}\label{main_eq}
T(t_1, \dots , t_6)=0\
\end{equation}
that such invariants satisfy, computed in \cite{hyp_mod_3}.   
The field of invariants $\S_8$ of binary octavics is $\S_8= k(t_1, \dots , t_6),$  where $t_1, \dots , t_6$ satisfies the equation $T(t_1, \dots , t_6)=0$. Hence, we have an explicit description of the hyperelliptic moduli $\H_3$; see \cite{hyp_mod_3} for details. 

Throughout this paper we will use the following important result

\begin{lem}[Shaska \cite{hyp_mod_3}]
Two genus 3 hyperelliptic curves $C$ and $C^\prime$, defined over an algebraically closed field $k$ of characteristic zero, with $J_2, J_3, J_4, J_5$ nonzero are isomorphic over $k$ if and only if \[ t_i (C) = t_i(C^\prime), \quad \textit{ for } \quad i=1, \dots 6.\]
\end{lem}

In the cases of curves when $t_1, \dots , t_6$ are not defined we will define new invariants as suggested in \cite{hyp_mod_3}.   From \cite[Lemma 4]{hyp_mod_3} we know that $J_2, \dots , J_7$ can't all be 0, otherwise the binary form would have a multiple root.


To describe the moduli points in cases when absolute invariants are not defined is not difficult.  In this case, one has to treat each case separately when any of the invariants $J_2, \dots J_5$ are zero.  Indeed, we can define invariants depending of which of the invariants is nonzero.  

If  $J_2 \neq 0$, then  we  define
\begin{equation}\label{j2=0}
 i_1= \frac {J_3^2} {J_2^3}, \quad  i_2= \frac {J_4}  {J_2^2}, \quad  i_3= \frac {J_5^2}  {J_2^5}, \quad  i_4= \frac {J_6}  {J_2^3}, \quad  i_5= \frac {J_7^2}  {J_2^7}, \quad i_6=\frac {J_8} {J_2^4} 
\end{equation} 
If $J_2 =0$ then we pick the smallest degree invariant among $J_3, \dots , J_7$ which is not zero.  
Let $J_2=0$ and  $J_3 \neq 0$.  Define 
\[ h_1= \frac {J_4^3}  {J_3^4 }, \quad  h_2=  \frac { J_5^3}  {J_3^5}, \quad  h_3= \frac {J_6}  {J_3^2}, \quad  h_4= \frac {J_7^3 }  {J_3^7}, \quad h_5 = \frac {J_8^3} {J_3^8}\]
Let $J_2=J_3=0$ and $J_4\neq 0$.  Then we have 
\[  j_1= \frac {J_5^4} {J_4^5}, \, \, j_2= \frac {J_6^2} {J_4^3}, \, \, j_3=\frac {J_7^4} {J_4^7}, \, \, j_4=\frac {J_8} {J_4^2}      \]
Let $J_2=J_3=J_4=0$ and $J_5 \neq 0$.  Then 
\[ k_1 = \frac {J_6^5} {J_5^6}, \, \, k_2= \frac {J_7^5} {J_5^7}, \, \, k_3= \frac {J_8^5} {J_5^8} \]
Let us assume that $J_2=J_3=J_4=J_5=0$.   In this case, we define the  absolute invariants 
\[ \t_1:= \frac {J_7^6} {J_6^7}, \quad \t_2 = \frac  {J_8^3} {J_6^4}, \]
There is only one curve in the case when $J_2= \dots = J_6 =0$, namely $ Y^2 = X^7-1$. 
In our discussion in section 5 we will see cases when $J_3=J_5=J_7=0$.  In this case we will use invariants defined in Eq.~\eqref{j2=0}.

Since a tuple $(t_1, \dots , t_6)$ uniquely determines the isomorphism class of a curve then we will study the locus of the curves with a fixed automorphism group $G$ in terms of such invariants $(t_1, \dots , t_6)$. The only interesting cases are  groups $G$ which have non-hyperelliptic involutions; see \cite{g-sh-s} or \cite{beshaj-2}.

To make it easier to state some of the results in the following sections we define the absolute invariants of $\X$ as follows

\[
\mathfrak p (\X) \, =  \, \left\{
\begin{split}
& (t_1, \dots , t_6) \quad \textit{if } \quad J_2, \dots J_5 \quad \textit{are nonzero } \\
& (i_1, \dots , i_6) \quad \textit{if     }  J_2\neq 0 \wedge \left( J_3=0 \vee J_4=0 \vee J_5=0    \right)    \\
& (h_1, \dots , h_5)    \quad \textit{if     }  J_2 = 0 \wedge   J_3 \neq 0\\
& (j_1, j_2, j_3, j_4) \quad \textit{if     }  J_2 = J_3= 0 \wedge J_4 \neq 0   \\
& (k_1, k_2, k_3) \quad \textit{if     }  J_2=J_3=J_4=0   \wedge  J_5 \neq 0  \\
& (\tau_1, \tau_2) \quad \textit{if     }  J_2=J_3=J_4=J_5=0 \\  
\end{split}
\right.
\]

\bigskip

For each case of the above we determine the equation of the moduli space $\H_3$.  For the first case this equation is Eq.~\eqref{main_eq}.  The rest of the cases are described briefly below. 
The $\mathfrak p (\X)$ determines uniquely the isomorphism class of $\X$.

%


Computational benefits of these invariants are that they are of small degree and therefore nicer especially when we deal with families of curves and have to compute symbolically. 
In each case for $\mathfrak p (\X)$ we can compute the equation of the moduli in terms of the corresponding invariants analogous to the $T(t_1, \dots , t_6)=0$ as in Eq.~\eqref{main_eq}.

\subsection{Computing the locus $\S$}

Let $\S$ denote the locus of genus 3 hyperelliptic curves with elliptic involutions.  It follows from the theory of Hurwitz spaces that this is an irreducible 3-dimensional subvariety of $\H_3$; see  \cite{issac, g-sh-s}. In \cite{g-sh-s} it is shown that $k(\S) = k(\s_2, \s_3, \s_4)$.  In this paper, we will give an explicit computational proof of this result and provide a birational parametrization of the locus $\S$.  We will outline all the computations and display only those results which are reasonable to display in this paper. 

Every genus 3 hyperelliptic curve which has and elliptic involution is isomorphic to a curve with equation as in Eq.~\eqref{eq3}.  The obvious strategy would be to compute the invariants $t_1, \dots , t_6$ in terms of $a, b, c$ and then eliminate $a, b, c$ from these equations.  This is rather difficult computationally.  Since dihedral invariants $\s_2, \s_3, \s_4$ are invariant under coordinate changes in $\P^1$ then we can express $t_1, \dots , t_6$ in terms of such invariants as stated in Theorem~\ref{sect1_thm}.   In the next few paragraphs we describe these computations. 

Let $\X_3$ be a genus 3 hyperelliptic curve with equation as in Eq.~\eqref{eq3}.  Notice that from the definitions of the dihedral invariants in Eq.~\eqref{eq2} we have 
\begin{equation}\label{subs}
b = \frac {\s_3} {a^2+c^2}, \quad  b^2 = \frac {\s_3^2} {\s_4 + 2 \s_2^2}, \quad a^8+c^8= \s_4^2 - 2 {\s_2^4}, \quad \left( a^2 + c^2 \right)^2 = \s_4 + 2 \s_2^2  
\end{equation}
We denote by $\l := a^2 + c^2 $.  Then,  $\l^2= \s_4 + 2 \s_2^2$ and $ \l a^2 = (a^2+c^2) \, a^2 = a^4 + \s_2^2$.    By changing the coordinate by 
\[ X \to \left( a^2+c^2 \right) \sqrt{a} \, X\]
we get the curve
\[ Y^2 = \l^8 a^4 \,  X^8 + \l^6  a^4  X^6 + \l^4 b a^2  \,  X^4 + \l^2  c  a \,  X^2 + 1 \]
Notice that the coefficient of $X^4$ is
\[ \l^4 b a^2 = (a^2+c^2)^2 \cdot \l^2 b a^2 = b (a^2+c^2) \cdot a^2 (a^2+c^2) \cdot \l^2 = \s_3 \cdot (a^4+\s_2) \cdot \l^2 \]
Then, we have the curve with equation  
\begin{equation}\label{eq_s}
 Y^2 = A \, X^8 + \frac {A} {\s_4 + 2\s_2^2} \, \, X^6  +  \frac {\s_3  ( A + \s_2^2) } {( \s_4 + 2\s_2^2 )^3} \, \, X^4  +  \frac { \s_2} {(\s_4 + 2\s_2^2)^3} \, \, X^2 + \frac {1} {(\s_4 + 2\s_2^2)^4} 
\end{equation} 
where $A=a^4$. Notice that by substituting $c= \frac {\s_2} a$ in the definition of $\s_4$ we get that $A+ \frac {\s_2^4} A = \s_4$,  which says that $A$ satisfies the equation 
\begin{equation} A^2 - \s_4 A + \s_2^4=0
\end{equation}
We will see how this equation will be useful when discussing the field of definition of the curve $\X_3$. 

From the Eq.~\eqref{eq3} we compute the invariants $J_2, \dots , J_8$.   By performing the above substitutions we get the following expressions
\begin{small}

\begin{equation}\label{eq_J}
\begin{split}
J_2  =  & 2 (140\,\s_4+280\,{\s_2}^2+5\,\s_2\s_4+10\,{\s_2}^3+{\s_3}^2)        \\  \\
J_3   = &  2 (6\,{\s_3}^3+525\,{\s_4}^2+2100\,\s_4{\s_2}^2+2100\,{\s_2}^4-55\,\s_2\s_3\s_4-110\,{\s_2}^3\s_3+1960\,\s_3\s_4+3920\,\s_3{\s_2}^2)  \\   \\
J_4   =& 2^6 \, (  {\s_3}^4+126\,\s_3{\s_4}^2+504\,\s_3\s_4{\s_2}^2+504\,\s_3{\s_2}^4+38416\,{\s_4}^2+153664\,\s_4{\s_2}^2+153664\,{\s_2}^4  \\
& -784\,{\s_2}^3\s_4-784\,{\s_2}^{5}+4\,{\s_4}^2{\s_2}^2+16\,\s_4{\s_2}^4+16\,{\s_2}^{6}-392\,\s_4{\s_3}^2-784\,{\s_2}^2{\s_3}^2+31\,\s_2\s_4{\s_3}^2\\  
& -196\,\s_2{\s_4}^2 +62\,{\s_2}^{3}{\s_3}^2  )  \\ \\
J_5   = & 2^5 \, (76832\,\s_3{\s_4}^2+307328\,\s_3{\s_2}^4+123480\,{\s_4}^2{\s_2}^2+246960\,\s_4{\s_2}^4+1148\,{\s_2}^4{\s_3}^2+287\,{\s_4}^2{\s_3}^2 \\
& -1568\,{\s_3}^3{\s_2}^2+41552\,{\s_2}^{5}\s_3-26\,{\s_2}^3{\s_3}^3-1680\,{\s_2}^{5}\s_4-208\,\s_3{\s_{2}}^{6}-140\,\s_2{\s_4}^3-840\,{\s_2}^3{\s_4}^2\\
& +20580\,{\s_4}^3+2\,{\s_3}^{5}-1120\,{\s_{2}}^{7}+307328\,\s_3\s_4{\s_2}^2+10388\,\s_2\s_3{\s_4}^2+41552\,{\s_2}^3\s_3\s_4 \\
& -52\,\s_3{\s_4}^2{\s_2}^2 +164640\,{\s_2}^{6}-784\,{\s_3}^3\s_4 -208\,\s_3\s_4{\s_2}^4+1148\,{\s_2}^2{\s_{3}}^2\s_4-13\,\s_2{\s_3}^3\s_4)   \\    \\
J_6   = & 2^9 \, \left( 2\,\s_2\s_4+4\,{\s_2}^3-196\,\s_4-392\,{\s_2}^2+{\s_3}^2 \right)  \left( {\s_3}^4-378\,\s_3{\s_4}^2-1512\,\s_3\s_4{\s_2}^2-1512\,\s_3{\s_2}^4 \right.  \\
&  -10192\,{\s_2}^3\s_4-10192\,{\s_2}^{5}-77\,\s_2\s_4{\s_3}^2-154\,{\s_2}^3{\s_3}^2+4\,{\s_4}^2{\s_2}^2+16\,\s_4{\s_2}^4+16\,{\s_2}^{6}\\
& \left. +38416\,{\s_4}^2 +153664\,\s_4{\s_2}^2  +153664\,{\s_2}^{4}-392\,\s_4{\s_3}^2-784\,{\s_2}^2{\s_3}^2  -2548\,\s_2{\s_4}^2  \right)   \\   \\
J_7 = & 2^7 \, ( 1120\,{\s_2}^4{\s_4}^3 -120472576\,\s_3{\s_2}^{6}-34300\,\s_2{\s_4}^4+3360\,{\s_2}^{6}{\s_4}^2+4480\,{\s_2}^{8}\s_4+140\,{\s_4}^4{\s_2}^2 \\
& -203840\,\s_3{\s_2}^{8}+608\,\s_3{\s_2}^{9}+4410\,\s_3{\s_4}^4-39200\,{\s_{{2}}}^{5}{\s_3}^3-931\,{\s_3}^4{\s_4}^2-3724\,{\s_3}^4{\s_2}^4\\
& -90\,{\s_3}^{5}{\s_2}^3-1176\,{\s_3}^{5}\s_4-2352\,{\s_3}^{5}{\s_2}^2+161896\,{\s_3}^2{\s_4}^3+8344\,{\s_2}^{7}{\s_3}^2-274400\,{\s_2}^3{\s_4}^3\\
& + 230496\,{\s_3}^3{\s_4}^2+921984\,{\s_3}^3{\s_2}^4+1295168\,{\s_2}^{6}{\s_3}^2+129077760\,{\s_2}^{6}\s_4-1097600\,{\s_2}^{7}\s_4  \\
& -15059072\,\s_3{\s_4}^3+96808320\,{\s_2}^4{\s_4}^2+29196160\,{\s_2}^{7}\s_3+2240\,{\s_2}^{10}+4033680\,{\s_4}^4 \\
& +2\,{\s_3}^{7}-90354432\,\s_3{\s_4}^2{\s_2}^2-180708864\,\s_3\s_4{\s_2}^4-270480\,\s_3{\s_2}^{6}\s_4-45\,{\s_3}^{5}\s_{{2}}\s_4\\
& +6258\,{\s_2}^{3}{\s_3}^2{\s_4}^2+912\,\s_3{\s_2}^{7}\s_4+971376\,{\s_2}^2{\s_3}^2{\s_4}^2-9800\,\s_2{\s_3}^3{\s_{{4}}}^2+345\,{\s_2}^2{\s_3}^3{\s_4}^2\\
& +456\,\s_3{\s_{2}}^{5}{\s_4}^2  -39200\,{\s_2}^3{\s_3}^3\s_4  +64538880\,{\s_2}^{8} +1380\,{\s_2}^{6}{\s_3}^3 \\
& -3724\,{\s_3}^4\s_4{\s_2}^2+921984\,{\s_3}^3\s_4{\s_2}^2+1043\,\s_2{\s_3}^2{\s_4}^3+21897120\,{\s_2}^3\s_3{\s_4}^2+76\,\s_3{\s_2}^3{\s_4}^3 \\
& +43794240\,{\s_2}^{5}\s_3\s_4+1380\,{\s_2}^4{\s_3}^3\s_4+3649520\,\s_2\s_3{\s_4}^3+1942752\,{\s_2}^4{\s_3}^{2}\s_4+980\,\s_3{\s_4}^3{\s_2}^2 \\
& +32269440\,{\s_4}^3{\s_2}^2 -99960\,\s_3{\s_{2}}^4{\s_4}^2  -548800\,{\s_2}^{9}  +12516\,{\s_2}^{5}{\s_3}^2\s_4 -823200\,{\s_2}^{5}{\s_4}^2 ) \\
\end{split}  
\end{equation}
\end{small} 

\medskip

We do not display  $J_8$ but it is  easy to compute as the previous ones.   The reader who want to obtain the above expressions can  use computational packages.  For example, "algsubs" would work in Maple and similar commands in other packages.  

Hence, we can write now $t_1, \dots , t_6$ in terms of the dihedral invariants $\s_2, \s_3, \s_4$ 
%

\begin{small}
\begin{equation}\label{eq_i}
\begin{split}
t_1 & =  \frac  1 2 \cdot \frac {(3920 \s_3 \s_2^2+2100 \s_4^2+2100 \s_2^2 \s_4+525 \s_2^4+7840 \s_3 \s_4+24 \s_3^3-110 \s_3 \s_4 \s_2-55 \s_3 \s_2^3)^2 (2 \s_4+\s_2^2)^2} {(560 \s_4+280 \s_2^2+10 \s_2 \s_4+5 \s_2^3+4 \s_3^2)^3} \\
t_2 & = \frac {64} {(560 \s_4+280 \s_2^2+10 \s_2 \s_4+5 \s_2^3+4 \s_3^2)^2 } \, (38416 \s_2^4+\s_2^6+4 \s_4^2 \s_2^2+4 \s_2^4 \s_4-392 \s_2^3 \s_4+153664 \s_2^2 \s_4-392 \s_4^2 \s_2 \\
& +4 \s_3^4 -98 \s_2^5+504 \s_3 \s_2^2 \s_4-1568 \s_3^2 \s_4-784 \s_3^2 \s_2^2+504 \s_3 \s_4^2+126 \s_3 \s_2^4+62 \s_3^2 \s_4 \s_2+153664 \s_4^2+31 \s_3^2 \s_2^3)) \\
t_3 &= - \frac {32}  {(560 \s_4+280 \s_2^2+10 \s_2 \s_4+5 \s_2^3+4 \s_3^2)    } \cdot \frac M N    \\
& M  = -123480 \s_4^2 \s_2^2-61740 \s_2^4 \s_4-307328 \s_3 \s_4^2-76832 \s_3 \s_2^4-5194 \s_3 \s_2^5+13 \s_3^3 \s_2^3+3136 \s_3^3 \s_4+1568 \s_3^3 \s_2^2 \\
 & -1148 \s_3^2 \s_4^2+280 \s_4^3 \s_2+420 \s_2^3 \s_4^2+210 \s_2^5 \s_4+13 \s_3 \s_2^6-287 \s_3^2 \s_2^4-10290 \s_2^6-8 \s_3^5-82320 \s_4^3+35 \s_2^7 \\
 & -307328 \s_3 \s_2^2 \s_4-20776 \s_3 \s_4^2 \s_2+52 \s_3 \s_4^2 \s_2^2-20776 \s_3 \s_2^3 \s_4+52 \s_3 \s_2^4 \s_4+26 \s_3^3 \s_4 \s_2-1148 \s_3^2 \s_4 \s_2^2 \\
& N  = 3920 \s_3 \s_2^2+2100 \s_4^2+2100 \s_2^2 \s_4+525 \s_2^4+7840 \s_3 \s_4+24 \s_3^3-110 \s_3 \s_4 \s_2-55 \s_3 \s_2^3 \\
t_4 & = - \frac A B \cdot  \frac {(8 (\s_2^3+2 \s_2 \s_4-392 \s_4-196 \s_2^2+2 \s_3^2))}  {(560 \s_4+280 \s_2^2+10 \s_2 \s_4+5 \s_2^3+4 \s_3^2) } \\
& A  = -38416 \s_2^4-\s_2^6-4 \s_4^2 \s_2^2-4 \s_2^4 \s_4+5096 \s_2^3 \s_4-153664 \s_2^2 \s_4+5096 \s_4^2 \s_2-4 \s_3^4+1274 \s_2^5+1512 \s_3 \s_2^2 \s_4\\
& +1568 \s_3^2 \s_4+784 \s_3^2 \s_2^2+1512 \s_3 \s_4^2+378 \s_3 \s_2^4+154 \s_3^2 \s_4 \s_2-153664 \s_4^2+77 \s_3^2 \s_2^3) \\
& B  = 38416 \s_2^4+\s_2^6+4 \s_4^2 \s_2^2+4 \s_2^4 \s_4-392 \s_2^3 \s_4+153664 \s_2^2 \s_4-392 \s_4^2 \s_2+4 \s_3^4-98 \s_2^5+504 \s_3 \s_2^2 \s_4\\
& -1568 \s_3^2 \s_4-784 \s_3^2 \s_2^2+504 \s_3 \s_4^2+126 \s_3 \s_2^4+62 \s_3^2 \s_4 \s_2+153664 \s_4^2+31 \s_3^2 \s_2^3 \\
\end{split}
\end{equation}

\end{small}

\noindent We are not displaying $t_5$ and $t_6$.

Hence, we have a  map give by the above equations 
\[
\begin{split}
& k^3\setminus \{ \Delta = 0\} \to \S =\M_3^b \cap \H_3 \\
& (\s_2, \s_3,  \s_4) \to (t_1, t_2, t_3, t_4, t_5, t_6 ) \\
\end{split}
\]
which as it will be shown in the next theorem is birational.

\begin{thm}
 $k(\S) = k(\s_2, \s_3, \s_4) $.  
\end{thm}

\proof
Since $k (\S)$ is a subfield of $k(\s_2, \s_3, \s_4$ which contains all $k(t_i)$ for $i=1, \dots , 6$ then $[k(\s_2, \s_3, \s_4):k(\S)]$ must divide each of degrees if $t_i$.  The degrees of $t_i$ as rational functions in $\s_2, \s_3, \s_4$ are respectively 12, 6, 7, 9, 10, 12.   
Hence, $[k (\s_2, \s_3, \s_4) : k (\S)]=1$.  This completes the proof. 

\qed 
 
 This was also proved in \cite{g-sh-s} for any genus $g \geq 2$.  Here we provide a direct computational proof and explicitly determine the formulas for $\s_2, \s_3, \s_4$ as rational functions in terms of $t_1, \dots , t_6$. 
We have the following theorem.

\begin{thm} \label{mainthm_kap3}
The space $\S:=\M_3^b \cap \H_3$ is an irreducible, codimension 1, rational subvariety of $\M_3^b$. Its defining equations   are 
\begin{equation}
\begin{split}\label{eq_L2_J}
 F_i (J_2, \dots , J_8 )= & 0, \quad i=1 \dots 5
\end{split}
\end{equation}
as displayed in \cite{homepage}.  The map  
\begin{equation}
\begin{split}
k^3\setminus \{\Delta=0\}   &   \to \S :=\M_3^b \cap \H_3 \\
(\s_2, \s_3, \s_4)          & \to \left( t_1, \dots , t_6\right) \\
\end{split}
\end{equation}
given by Eq.~\eqref{eq_i}  (in homogeneous coordinates  by the formulas \eqref{eq_J})  is birational and surjective.
\end{thm}

\proof  From Theorem 1, ii) we have that  $\M_3^b\cap \H_3$ is birationally isomorphic to the coarse moduli space of smooth curves of genus 2 together with a nontrivial divisor class of order 2.  Since this space is an irreducible, 3-dimensional, rational variety then the first part of the theorem is proved.

It remains to show  the map in the Theorem is birational.
We need to show that the degree of the field extension  $k(\s_2, \s_3, \s_4)/k(\S)$ is 1. For
this we use the functions $ t_1, t_2, t_3, t_4, t_5, t_6$ in $k(\H_3)$.    The fact that $[k(\s_2, \s_3, \s_4) : k(\S) ] =1$   comes straight from the computation of the locus $\S$ where we get rational expressions for the $\s_2, \s_3, \s_4$ or from Theorem 2.  This completes the proof. 
\qed

\subsection{Field of moduli vs field of definition}
It is a classical problem in the arithmetic of the algebraic curves to try to find an equation of the curve in terms of the moduli point corresponding to this curve.  In other words, this means that given the moduli point $\p (\X)$, could we determine an equation for $\X$ in terms of the coordinates of $\p (\X)$.  If an equation of the curve can be found in terms of coordinates of the moduli point we say that the field of moduli is the same with the minimal field of definition.  In 2003 the first author conjectured that this would be the case for all hyperelliptic curves with extra involutions \cite{sh_03, sh_02}.  It is true for $g=2$ and as we will see next it is true for all genus 3 hyperelliptic curves in $\p \in \M_3^b \cap \H_3$ such that $|  \Aut (\p) | > 2$. There have been claims on whether the above conjecture is true or false and some confusion from work of Huggins \cite{Hug} and Fuertes \cite{Fuertes}  which seem to come from different definitions of the field of moduli. 


Summarizing we have the results of section 4.1 and Eq.~(15) and (16) we have the following:
\begin{proposition} Let $[\X ] \in \M_3^b \cap \H_3$.  Then the following hold true:

i)   $\X$ is isomorphic to a curve with equation 
\[ Y^2 = A \, X^8 + \frac {A} {\s_4 + 2\s_2^2} \, \, X^6  +  \frac {\s_3  ( A + \s_2^2) } {( \s_4 + 2\s_2^2 )^3} \, \, X^4  +  \frac { \s_2} {(\s_4 + 2\s_2^2)^3} \, \, X^2 + \frac {1} {(\s_4 + 2\s_2^2)^4} \]
where $A$ satisfies
\begin{equation}\label{quadratic}  A^2 -  \s_4 A + \s_2^4 =0, \end{equation}
for some $(\s_2, \s_3, \s_4 \in  k^3 \setminus  \{ \D_{\s_2, \s_3, \s_4 } = 0 \}$  and $\D_{\s_2, \s_3, \s_4 }$ as in Eq.~(6). 

ii) Let $F$ denote the field of moduli of $\X$ and $F^\prime$ its minimal field of definition.  Then, $F^\prime \subset F(A)$ and $[F^\prime : F] \leq 2$. An equation of $\X$ over $F(A) $ is given by Eq.~\eqref{quadratic}.

iii) If the discriminant $  d= \s_4^2 - 4 \s_2^4  $ of the quadratic  in Eq.~ \eqref{quadratic}   is 
 is a complete square in $k (\s_2, \s_3, \s_4)$ then the corresponding curve is defined over its field of moduli. 
\end{proposition}

\proof  Part i) was proved in section 4.1.  The field of moduli of a curve $\X$ is $F=k(t_1, \dots , t_6)$. Hence, form Theorem 2, $F=k( \s_2, \s_3, \s_4)$.  Since $\X$ is defined via Eq.~\eqref{quadratic} over 
$F \left(\sqrt{\s_4^2 - 4 \s_2^4} \right)$, then $F^\prime \subset F(A)$ and $[F^\prime : F] \leq 2$.
Prat iii) is clearly true. 

\qed

The above result improves on the bound of the degree of $[F^\prime : F] \leq 8 $ as shown in \cite{ritz1}. It is expected that the field of moduli is a field of definition for all curves $\X \in \M_3^b \cap \H_3$.  A generalization of methods for $g=2$ should provide and algorithm also for $g=3$. 

 
\section{Singular locus of $\S$, classification of strata of the hyperelliptic moduli}

In this section we give a classification of the strata of the hyperelliptic moduli 3.  The stratum of the hyperelliptic moduli has been known to the classical algebraic geometers.  Indeed, it is the only case that was considered fully known even though explicit descriptions were not available even for small genus (i.e., $g=2$).  On the turn of the new century a couple of papers appeared for the case of genus $g=2$; see \cite{gaudry} and \cite{sh-v}.

\subsection{Singular loci in terms of dihedral invariants} 

Next, we will characterize each one of these loci in terms of the $\ss$-invariants.   The proof of the following theorem can be found in \cite{g-sh-s}, where such relations were determined by studying the group action on the invariants $\s_2, \s_3, \s_4$.  Here we give a more computational proof.

\subsection{2-dimensional strata}

There are two 2-dimensional loci in $\H_3$ which correspond to the case when the reduced automorphism group of the curve is isomorphic to $V_4$.  
Indeed, the following is true for any genus $g \geq 2$; see \cite[Theorem 3]{g-sh}.  

Let $\X_g $ be a genus $g$ hyperelliptic curve with an extra involution, $\bAut (\X_g)$ its reduced automorphism group,  and $(\s_4, \dots , \s_g)$ its corresponding dihedral invariants as defined in \cite{g-sh}.  

If $V_4\emb     \bAut (\X_g)$ then   $ 2^{g-1}\, \u_1^2 = \u_g^{g+1}$.    Moreover, if $g$ is odd then $V_4 \emb \bAut( \X_g )$ implies that
\begin{equation} \left( 2^r \, \u_1 - \u_g^{r+1}\right) \, \left( 2^r \, \u_1 + \u_g^{r+1}\right)=0  \end{equation}
where $r= \left[ \frac {g-1} 2 \right]$. The first factor corresponds to the case when  involutions
of $V_4\emb \bAut (\X_g)$ lift to involutions in $\Aut (\X_g)$, the second factor corresponds to the case when two
of the involutions of $V_4\emb \bAut (\X_g)$ lift  to elements of order 4 in $\Aut (\X_g)$.

In the case of genus $g=3$,  we have $\u_1=\s_4$ and $\u_g=\frac {\s_2} 2$ and the relation becomes 
 $\s_4- 2 \s_2^2=0$.   This gives exactly the cases when the full automorphism group is $\Z_2^3$.  We will verify such fact directly below.

\subsubsection{The automorphism group is isomorphic to  $G \iso \Z_2^3$:}  

The equation of the curve from Table 3 of \cite{kyoto} is 
\[ y^2 = (x^4+ax^2+1)(x^4+b x^2+1)\]
Let new parameters $u$ and $v$ be as follows
\[ u:=a+b, \quad and \quad v=ab. \]
Then we have 
\begin{equation}\label{eq_Z2_3}
 y^2= x^8 + ux^6 + (v+2) x^4 + u x^2 +1 
\end{equation} 
The dihedral invariants are 
\[ \s_4 = 2u^4, \quad \s_3 =2 u^2 (v+2), \quad \s_2 =  u^2\]
Then directly we can verify  $u^2=  \s_2$, $v= \frac {\s_3 - 2\s_2} {2\s_2}$ and  
\begin{equation}
\s_4- 2 \s_2^2=0.
\end{equation}
By transforming the coordinate $X$ as $X\to \sqrt{u} X$ on the curve in Eq.\eqref{eq_Z2_3} we get  
\[ Y^2 = u^4 X^8 + u^4 X^6 + (v +2)(u^2)X^4 + u^2 X^2 +1 \]
or
\begin{equation} Y^2 = \s_2^2 \,  X^8 + \s_2^2 \, X^6 + \frac {\s_3} {2} \, X^4 + \s_2 \, X^2 + 1 \end{equation}
We compute the invariants $t_1, \dots , t_6$  in terms of $\s_3, \s_2$.   Eliminating $\s_3$ and $\s_2$ from the system of equations gives the locus $\S (\Z_2^3)$ and rational expressions of $\s_3, \s_2$ in terms of $t_1, \dots , t_6$.  The computations are long and the results involve very large expressions.  Instead, we provide a quicker proof for the reader which is easier to check.

From the expressions of $t_1, \dots , t_6$ in terms of $\s_3, \s_2$ we eliminate $\s_3$.  In other words, $\s_3$ is easily written as a rational function in terms of $\s_2, t_1, \dots t_6$.   Any software package such as maple or Mathematica will be able to do this.  We are left with 5 equations of degree 20, 18, 16, 23, and 21.  Since the function field $k(\S(\Z_2^3))$ is a subfield of $k(\s_2)$ then the degree of this extension  $[k(\s_2), k(\S(\Z_2^3))]$ is a common divisor of 20, 18, 16, 23, and 21.  Therefore, $[k(\s_2), k(\S(\Z_2^3))]=1$  and $k(\S(\Z_2^3))=k(\s_2)$.

\begin{lem} Every genus 3 hyperelliptic curve with full automorphism group isomorphic to  $\Z_2^3$ has equation 
\begin{equation} Y^2 = \s_2^2 \,  X^8 + \s_2^2 \, X^6 + \frac {\s_3} {2} \, X^4 + \s_2 \, X^2 + 1  \end{equation}
for $\s_3, \s_2  \neq 0, 4$.   Moreover,  $k(\S (\Z_2^3) = k(\s_4, \s_3)$ and therefore the field of moduli is a field of definition.  
\end{lem}

The result of the Lemma above can be obtained directly by the $V_4$-locus, by enforcing the equation $2\s_4-\s_2^2=0$.   The expressions for $\s_3, \s_2$ and the equations of $\S(\Z_2^3)$ in terms of $t_1, \dots , t_6$ are displayed in \cite{homepage}. 

\subsubsection{The automorphism group is isomorphic to $G \iso \Z_4$:}  

This is the only case that is not a sublocus of the space $\S$.  The equation of the curve from Table 3 of \cite{kyoto} is  
\[ y^2 = x(x^2-1)(x^4+ax^2+b))\]
In this case we have $J_3=J_5=J_7=0$.   The defining equations of this space are two polynomials in $J_2, J_4, J_6, J_8$.  We make them part of the "genus3" package in \cite{homepage}.  It is worth noting that both $a$ and $b$ can be expressed as rational functions in $t_1, \dots , t_6$. Hence, in this case the field of moduli is a field of definition.

\subsection{1-dimensional strata}

\subsubsection{The automorphism group is isomorphic to  $\Z_2\times D_8$:}  
Given a curve $C$ in the $\Z_2\times D_8$ locus, from Table 1, it has equation:
\begin{equation}\label{Z2xD8}
Y^2=X^8+aX^4+1,
\end{equation}
where $a \neq \pm 2$.
We calculate the  invariants, $t_1, t_2, \dots, t_6$ and denote $t:=a^2$. Then we have
\[ 
\begin{split}
t_1 =2\,  \frac {t \left( 3\,t+980 \right)^2}{ \left( 140+t \right)^3 },  \qquad  t_2 =16\,{\frac { \left(t -196 \right)^2} { \left( 140+t \right)^2} },  \qquad   t_3 = 8\,  \frac { \left(t -196 \right)^2} { \left( 140+t \right)  \left(  3\,t+980 \right) },   \\
t_4 =4\, \frac {-196+t} {140+t}, \qquad  t_5 =4\,   \frac {t-196} {140+t},  \qquad t_6 =128\,   \frac { \left( 9\,t+980 \right)  \left( t-196 \right)^3 } { \left( 140+t \right)^4 }    \\
\end{split}      
\]
These invariants are not defined for $t=- 140$ and $t= - \frac {980} 3$.  We can rewrite the above equations as 
\[ t = -28 \frac {5 t_4+28} {t_4-4}\]
and the equations of this submoduli space are given by 
\[  t_1 = - \frac {175} {288} t_4^2+ \frac {125} {3456} t_4^3+\frac {686} {27}, \, t_2 = t_4^2, \, t_3 = -6 \frac {t_4^2} {(5 t_4-56)}, \, t_5 = t_4,\, \,   t_6 = \frac {49} 3 t_4^3+\frac 5 {12} t_4^4 \]
Notice that we can get the equations in terms of $J_2, \dots , J_8$ very easily by substituting $t_1, \dots , t_6$.  Such equations would be valid in even in the  cases when some of $J_i$ are zero. 

The equation of the curve can be written in terms of $t$ as follows.  Let $X\to a^{\frac 1 4} X$.  Then the equation of the curve becomes 
\[ Y^2 = t X^8 + tX^4 +1 \]
Hence, for this family of curves the field of moduli is a field of definition. 

\begin{lem} Every genus 3 hyperelliptic curve with full automorphism group isomorphic to  $Z_2 \times D_8$ has equation 
\begin{equation} Y^2 = t X^8 + tX^4 +1 \end{equation}
for some $t \neq 0, 4$.   If $t\neq  -140, - \frac {980} 3$ then  \[t = -28 \frac {5 t_4+28} {t_4-4}\] in terms of the absolute invariants and therefore the field of moduli is a field of definition.  
\end{lem}

\subsubsection{The automorphism group is isomorphic to  $D_{12}$:}  
The equations of the curve is:
\[Y^2=X(X^6+aX^3+1)\]
We perform the following coordinate change $X  \to a^{\frac 1 3} X$ and the equation of the curve becomes
\[ Y^2 = X (t X^6 + tX^3 +1,\]
where $t = a^2$.  Then, we have 
\[ 
\begin{split}
 t_1 = 9\,{\frac {t \left( 4\,t+245 \right)^2}{ \left( -35+2\,t \right)^3}},    \qquad t_2 = {\frac { \left( 8\,t+49 \right)^2}{ \left( -35+2\,t \right)^2}},   \qquad t_3 = \frac 1 3\,{\frac { \left( 8\,t+49 \right)^2}{ \left( -35+2\,t \right)  \left( 4\,t+245 \right) }} \\
 t_4 = {\frac {8\,t+49}{-35+2\,t}},    \qquad  t_5 = {\frac {8\,t+49}{-35+2\,t}},  \qquad  t_6 = 3\,{\frac { \left( 12\,t+245 \right)  \left( 8\,t+49 \right)^3}{ \left( -35+2\,t \right)^4}} \\
\end{split}      
\]
We can eliminate $t$ 
\[ t= \frac 7 2 \,{\frac {5\,t_4 +7}{t_4-4}},\] 
and the equations for the submoduli space become 
\begin{equation}\label{d_12_eq}  
t_1={ \frac {686}{27}}+{\frac {125}{54}}\,{t_4}^{3}-{\frac {175}{18}} \,{t_4}^{2}, \, 
t_2 ={t_4}^{2}, \, 
t_3=  \frac {t_4^2}   {5\,t_4-14}, \, 
t_5=t_4, \,
t_6={\frac {65}{9}}\,{t_4}^{4}-{\frac {98}{9}}\,{t_4}^{3}  
\end{equation}
\begin{lem} Every genus 3 hyperelliptic curve with full automorphism group isomorphic to  $D_{12}$ has equation 
\begin{equation} Y^2 = X(t X^6 + tX^3 +1 ) \end{equation}
for $t \neq 0, 4$.   If $t\neq  - \frac {35} 2, - \frac {245} 4$ then  \[t = \frac 7 2 \,{\frac {5\,t_4 +7}{t_4-4}} \] in terms of the absolute invariants and therefore the field of moduli is a field of definition.  
\end{lem}

\subsubsection{The automorphism group is isomorphic to  $\Z_2\times\Z_4$:}  
The equations of the curve is:
\[y^2=(x^4-1)(x^4+ax^2+1)\]
By a transformation $X \to a^{\frac 1 2} X$ the equation of the curve becomes 
\[ Y^2 \, = \, (tX^4-1)(tX^4+tX^2+1)\]
Since this curve has an element of order 4 and therefore a factor of $X^4-1$ then $J_3=J_5=J_7=0$.  In this case the absolute invariants $t_1, \dots t_6$ are not defined.  Hence we use the invariants $i_1, \dots , i_5$ as in Eq.\eqref{j2=0}.
We have $i_1, i_3, i_5=0$ and 
\[ 
\begin{split}
i_2  & ={\frac {64}{25}}\,{\frac {{t}^{2}+9604-49\,t}{ \left( 28+t \right) ^{2}}}, \\
 i_4  & = {\frac {512}{125}}\,{\frac { \left( t-98 \right)  \left( {t}^{2}-637\,t+9604 \right) }{ \left( 28+t \right) ^{3}}}, \\
i_6   & =-{\frac {512}{125}}\,{\frac {11\,{t}^{4}-12397\,{t}^{3}+1296540\,{t}^{2}+368947264-43294832\,t}{ \left( 28+t \right) ^{4}}} \\
\end{split}
\]
Hence, we get
\[  t=28\,{\frac {15625\,i_2 \, i_4 -152500\, i_4 +24375\,{i_2}^{2}+1215200\,i_2-2809856}{-15625\,i_2 \,i_4 +2500\, i_4 +245625\,{i_2}^{2}-725600\,i_2+401408}} \]
and
\begin{small}

\begin{equation}
\left\{
\begin{split}
&   -81462500\,i_4+927746400\,i_2-963780608-256055625\,{i_2}^{2}-1953125\,{i_4}^{2}\\
&  +36093750\,i_2 \,i_4 +15187500\,{i_2}^{3}   =0 \\
& -22689450000\,i_6-4593393436800\,i_2+4628074479616+52734375 \,{i_6}^{2}+8912109375\,{i_2}^{4}\\
& +1371093750\,{i_2}^{2} i_6 +5788125000\,i_2 \, i_6 +1572126780000\,{i_2}^{2}-215275375000\,{i_2}^{3}   =0 \\
\end{split}
\right.
\end{equation}
 
\end{small}

\subsection{0-dimensional strata}

We first briefly go over the 0-dimensional cases.  

\subsubsection{Case : $G \iso \Z_2\times S_4$:}  
The equation of the curve is  $y^2=x^8+14x^4+1$ and its absolute invariants are 
\[  
(t_1,t_2,t_3,t_4,t_5,t_6) = \left(\frac{15435}{8}, \frac{784}{25}, \frac{56}{25}, \frac{-28}{5}, \frac{28}{5}, \frac{7760032}{125} \right)
\]

The next two cases correspond to curves with $ J_3=J_5=J_7=0$. In both cases we use invariants $i_2, i_4, i_6$ as in Eq.~\eqref{j2=0}.

\subsubsection{Case : $G\iso U_6$:}  
The equations of the curve is given by   $y^2=x(x^6-1)$ and its absolute invariants are 
 $i_1=i_3=i_5=0$ and \[ i_2= \frac {49} {25}, \quad i_4= - \frac {343} {125}, \quad i_6= \frac {7203} {125} \]

\subsubsection{Case : $G \iso V_8$:}  
The equations of the curve is $y^2=x^8-1$ and its absolute invariants are 
 $i_1=i_3=i_5=0$ and \[ i_2= \frac {784} {25}, \quad i_4= - \frac {21952} {125}, \quad i_6= - \frac {307328} {125} \]

The following theorem determines relations among $\s_2, \s_3, \s_4$ for each group $G$ such that $V_4 \emb G$. 

\begin{thm}\label{thm_last}
Let $\X$ be a  curve in   $\S= \M_3^b \cap \H_3$. Then, one of the following occurs:


i) $\Aut (\X) \iso \Z_2^3$ if and only if   $\s_4 - 2 \s_2^2 =0 $

ii) $\Aut (\X) \iso \Z_2 \times D_8$ if and only if     $\s_2=\s_4=0$

iii) $\Aut (\X) \iso \Z_2 \times \Z_4$ if and only if   $\s_4+ 2 \s_2^2=0$ and $\s_3 =0$.

iv) $\Aut (\X) \iso D_{12}$ if and only if  
\begin{equation}\label{d_u_12}
\begin{split}
\s_3 & =  {\frac {1}{75}}\, \left( 9\, \s_2 -224 \right)  \left( \s_2-196 \right) \\
\s_4 & = -{\frac {9}{125}}\,{\s_2}^{3}+{\frac {1962}{125}}\,{\s_2}^{2}-{\frac {840448}{1125}}\,\s_2+{\frac {9834496}{1125}}    \\
\end{split}
\end{equation} 

\end{thm}

\begin{proof} Part i) and ii) are immediate consequences of the previous discussions. 
For part iii), 
we start with the  curve $\X$ with equation \[Y^2 = (X^4 - 1)(X^4+ aX^2 +1).\] Transforming  $X \to \epsilon_{16}X$ we have
\[Y^2=X^8-\epsilon_{16}^6 a X^6 + \epsilon_{16}^2 a X^2 + 1,\]
where $\e_{16}$ is the 16-th root of unity. 
The dihedral invariants are 
\[ \s_2= a^2, \quad \s_3 =  0, \quad \s_4 - 2 a^4 \]
By eliminating $a^2$ we have that
\[  \s_4 + 2\s_2^2 =0, \quad \textit{  and } \quad \s_3 =0\]

Conversely, if the above equations hold then $a^2+c^2=0$.  Take a curve with equation as in Eq.~(4) and compute $i_2, i_4, i_6$.  These invariants satisfy Eqs.~(30).  Hence, the curve is in the $(\Z_2 \times \Z_4)$--locus.

For case $iv)$, let $\X$ be a curve with equation $Y^2=X \, (X^6 + a X^3 +1)$, where $a\neq 0,\pm 2$. By a transformation $X \to \frac {X+1} {X-1}$, $\X$ has equation
\[Y^2= X^8 + ( 5 - 9\l) X^6 + 3 (\l +1)  X^4 + (5 \l -9) X^2 +\l\]
where $\l=\frac {a-2}{a+2}$, $\l\neq 0, \pm 1$. Then, by another transformation $X \to \sqrt[8]{\l}\,X$, we get the following curve 
\[Y^2= X^8 + \frac {( 5 - 9\l)}  { \l^{\frac 1 4} }  X^6 + 3 \frac {(\l +1)} {\l^{\frac 1 2}}  X^4 + \frac {(5 \l -9)} {\l^{\frac 3 4}} X^2 + 1 \]
Computing the dihedral invariants:
%
\begin{equation*}
\begin{split}
\s_2 & =  \frac 1 {\l} \, (5-9\l) (5\l-9)   \\
\s_3 & = \frac 3 {\l} \, (\l + 1) \, \left[ (5-9\l)^2 + \frac 1 {\l} (5\l-9)^2    \right]      \\
 \s_4 & = \frac 1 {\l} \, (5 - 9\l)^4 + \frac 1 {\l^3} (5\l -9)^4 \\
\end{split}
\end{equation*}
%
Eliminating $\l$, we get $\l = \frac {45} {106-\s_2}$ 
and
\begin{equation}\label{d_u_12}
\begin{split}
\s_3 & =  {\frac {1}{75}}\, \left( 9\, \s_2 -224 \right)  \left( \s_2-196 \right) \\
\s_4 & = -{\frac {9}{125}}\,{\s_2}^{3}+{\frac {1962}{125}}\,{\s_2}^{2}-{\frac {840448}{1125}}\,\s_2+{\frac {9834496}{1125}}    \\
\end{split}
\end{equation}

Conversely, let us assume that the equations in \eqref{d_u_12} hold.  From the expressions of $t_1, \dots , t_6$ in Eqs.~(18) and equations in \eqref{d_u_12} we eliminate $\s_2, \s_3, \s_4$ 
to get the equations of the $D_{12}$-locus in Eq.~\eqref{d_12_eq}.  The proof is complete.
\end{proof}

\bibliographystyle{amsplain}

\end{document}